\documentclass{amsart}
\usepackage{graphicx}
\usepackage{amssymb}
\newtheorem{thm}{Theorem}[section]
\newtheorem{cor}[thm]{Corollary}
\newtheorem{lem}[thm]{Lemma}
\newtheorem{prop}[thm]{Proposition}
\theoremstyle{definition}

\theoremstyle{question}
\newtheorem{que}[thm]{Question}
\numberwithin{equation}{section}
\begin{document}

\title[non-nilpotent graph of a group]{non-nilpotent graph of a group}%
\author{Alireza Abdollahi and Mohammad Zarrin}%
\address{Department of Mathematics, University of Isfahan, Isfahan
 81746-73441, Iran}%
 \email{a.abdollahi@math.ui.ac.ir,  zarrin1380@yahoo.com}
 %\thanks{}%
%\date{}%
%\dedicatory{}%
%\commby{}%
% ----------------------------------------------------------------
\begin{abstract}We associate a graph $\mathcal{N}_{G}$ with a  group
 $G$ (called the non-nilpotent graph of $G$)
as follows: take $G$ as the vertex set  and   two  vertices are adjacent   if they  generate a non-nilpotent subgroup. In this paper
we study the graph theoretical properties of $\mathcal{N}_{G}$ and
its induced subgraph  on $G\backslash nil(G)$, where $nil(G)=\{x\in
G \;|\; \langle x,y\rangle\ \text{is nilpotent for all}\; y\in
G\}$. For any finite group $G$, we prove that $\mathcal{N}_G$ has either $|Z^*(G)|$ or $|Z^*(G)|+1$ connected components,
 where $Z^*(G)$ is the hypercenter of $G$.
  We give a new characterization for finite nilpotent groups
in terms of the non-nilpotent graph. In fact we prove that a
finite group $G$ is nilpotent if and only if the set of vertex
degrees of $\mathcal{N}_G$ has at most two elements.
  \end{abstract}
\maketitle
% ----------------------------------------------------------------
\section{\textbf{Introduction and results}}
Study of algebraic structures by  graphs associated with them
gives rise to many recent and interesting results (see e.g.,
\cite{Ab1}, \cite{Ab2}). Here we associate a simple graph with
a group by the property of being nilpotent.  P. Erd\"os has introduced a graph associated with
a group by the property of commutativity: it is a graph whose vertex set is the set of elements
 of the group and two vertices are adjacent whenever they do not commute (see \cite{ne} and its
  review by  Erd\"os in Mathematical Review MR0419283 (54\#7306)).
Let $G$ be a group. We associate a simple graph $\mathcal{N}_G$
with $G$ \rm(called the non-nilpotent graph of $G$\rm) as follows:
take $G$ as the vertex set and join two vertices by an edge if they do
not generate a nilpotent subgroup. Our aim is to study a group $G$
by the information stored in its non-nilpotent graph and to study how the graph theoretical
properties  of $\mathcal{N}_G$  effect on the group ones of $G$.
Note that if
$G$ is weakly nilpotent (i.e., every two generated subgroup of
$G$ is nilpotent),  $\mathcal{N}_G$ has no edge. It follows
that the non-nilpotent graphs of weakly nilpotent groups with the
same cardinality are isomorphic. So we must be  interested
in non weakly nilpotent groups. On the other hand, as there are
  vertices in $\mathcal{N}_G$ which are isolated
(e.g., elements in the hypercenter of $G$),  it is wise to
study the subgraph $\mathfrak{N}_G$ of $\mathcal{N}_G$ induced
by $G\backslash nil(G)$, where $$nil(G)=\{x\in G \;|\; \langle
x,y\rangle \;\text{is nilpotent for all}\; y\in G\}.$$ It is not
known whether the subset $nil(G)$ is a subgroup of $G$, but in
many important cases it is a
subgroup. In particular, $nil(G)$ is equal to the hypercenter $Z^*(G)$ of $G$ whenever $G$ satisfies the maximal condition
on its subgroups or $G$ is a finitely generated solvable group (see Proposition \ref{l0}, below).

 The set of  adjacent vertices to a vertex $v$ in a graph $\Gamma$ is denoted by $N_{\Gamma}(v)$. For
  an element $x$ of a group $G$, the set $N_{\overline{\mathcal{N}_G}}(x)\cup \{x\}$ will be denoted by
   $nil_G(x)$, where $\overline{\mathcal{N}_G}$ is the complement of the non-nilpotent graph of $G$. In
   fact $$nil_G(x)=\{g\in G \;|\;\langle x, g \rangle \;\; \text{is nilpotent} \}.$$ We call $nil_G(x)$ the
    {\sl nilpotentizer} of $x$ in $G$.  In Section $2$ we show some results on the  nilpotentizers.\\

In Section $3$ we give some properties   of
$\mathfrak{n}$-groups, where by an $\mathfrak{n}$-group we mean a group $G$ in which  $nil_G(x)$ is a subgroup  for every $x\in G$. We
also study some possible relations between simple $\mathfrak{n}$-group and  some
classes of groups (see Theorem \ref{tt}, below).

A set $C$ of vertices of a graph $\Gamma$ whose induced subgraph
is a complete subgraph is called  a clique and  the maximum size
(if it exists) of a clique in a graph is called the clique number
of the graph and it is denoted by $\omega(\Gamma)$.
 In Section $4$  we state some previously proven results about groups whose
non-nilpotent graphs have finite clique number. We prove some results on the girth and diameter of the non-nilpotent
 graph also we show that there is no finite  non-nilpotent  group $G$ such that $\mathfrak{N}_G\cong \mathfrak{N}_H$
  or $\mathfrak{N}_{G/N}$ for some proper subgroup $H$ or quotient $G/N$ of $G$.\\

In Section $5$ we shortly prove the connectedness of $\mathfrak{N}_G$ for finite non-nilpotent groups $G$. Also, in
 this case, we estimate the diameter of $\mathfrak{N}_G$ to be at most $6$ while we conjecture that the best bound
 for the diameter must be $2$.\\

A planar graph is a graph that can be embedded in the plane so
that no two edges intersect geometrically except at a vertex which
both are incident. In Section 6  we show that for a non-nilpotent finite group $G$, $\mathfrak{N}_{G}$ is planer if and only if $G\cong S_3$.\\

The degree of a vertex $v$ in a graph $\Gamma$ is the number of vertices which are adjacent to $v$ in $\Gamma$.  A graph $\Gamma$ is said to have
$m$ kind vertex degrees if the
set of vertex degrees of $\Gamma$ is of size $m$. So by definition,  a graph is regular if and only if it has 1 kind vertex degree. In Section 7
we prove that a finite group $G$ is nilpotent if and only if $\mathcal{N}_G$ has at most two kind vertex degrees. This shows that there is no
finite non-nilpotent group $G$ such that  $\mathfrak{N}_G$ is regular.
\section{\textbf{Some properties of nilpotentizer }}
Let $G$ be a group and $x$ be an element of $G$.
The set of vertices of the complement graph $\overline{\mathcal{N}_G}$ of the non-nilpotent graph $G$ which are adjacent to $x$ together with
$x$ itself consist a set which we call it the {\em nilpotentizer} of $x$ in $G$. It will be denoted by $nil_G(x)$.
Indeed
$$nil_G(x)=\{g\in G\mid\langle x,g\rangle ~~\textrm{is nilpotent} \}$$
  For a nonempty subset $S$ of $G$, we define the
nilpotentizer of $S$ in $G$ as
$$nil_G(S)=\bigcap_{x\in S} nil_G(x).$$  We call $nil_G(G)$ the
nilpotentizer of $G$, and it will be  denoted  by $nil(G)$. Thus
$$nil(G)=\{x\in G \mid\langle x, y \rangle ~~\textrm{is nilpotent for all}\; y\in G\}.$$
 Besides nilpotentizers of elements in a group which are not necessarily subgroups even in finite solvable groups (see Lemma \ref{s4}, below),
  there are vast classes of groups, in which   nilpotentizers of groups are subgroups. In Section 3, we will study $\mathfrak{n}$-groups, that
   is, groups in which the nilpotentizer of each element is a subgroup.
\begin{prop}\label{l0}
Let  $G$ be  any group. Then
\begin{enumerate}
\item $Z^{*}(G)\subseteq nil(G) \subseteq R(G)$, where $R(G)$ is the set of right Engel elements of $G$.
\item If $G$ satisfies the maximal condition on its subgroups or $G$ is finitely generated solvable group
then $Z^{*}(G)=nil(G)=R(G)$.
\end{enumerate}
\end{prop}
\begin{proof}
(1) \; It is straightforward.\\
(2) \; It follows from \cite[Theorem 12.3.7]{d.j.r} and the main result of \cite{Br}.
\end{proof}
\begin{lem}\label{ll1}
Let $G$ be any group,  $N$  a normal
subgroup of $G$ and $x,g\in G$. Then
\begin{enumerate}
\item $nil(G)\subseteq nil_G(x)$. \item
$\frac{nil_G(x)N}{N}\subseteq nil_{\frac{G}{N}}(xN)$.
{\rm(}obviously here $\frac{nil_G(x)N}{N}:=\{yN \;|\; y\in
nil_G(x)N\}${\rm)}. \item $nil_{\frac{G}{K}}(xK)=
\frac{nil_G(x)}{K}$, where $K$ is a normal subgroup of $G$ with
$K\leq Z^*(G)$. \item $G$ is an  $\mathfrak{n}$-group if and only
if $\frac{G}{K}$ is an $\mathfrak{n}$-group for some normal
subgroup $K$ of $G$ with $K\leq Z^*(G)$. \item If $\langle
x\rangle=\langle y\rangle$ for some $y\in G$, then
$nil_G(x)=nil_G(y)$. \item  $nil_G(x)^g=nil_G(x^g)$.
\end{enumerate}
\end{lem}
\begin{proof}
The proof is straightforward.
\end{proof}
  \begin{lem}\label{l9}
Let $G$ be a finite group and $a$ be an element of
$G$. Then  $|nil_G(a)|$ is divisible by $|a|$.
\end{lem}
\begin{proof}
 It is easy to see that $nil_G(a)$ is the union of  maximal
nilpotent subgroups of $G$  containing $a$. On the other hand,
each of these subgroups is a union of cosets of $\langle
a\rangle$. It follows that
the order of  $a$ divides  $|nil_G(a)|$.
 \end{proof}
 \section{\textbf{Groups in which nilpotentizers of elements are subgroups}}
 We  call a group $G$ is an $\mathfrak{n}$-group if
$nil_G(x)$ is a subgroup of $G$ for every $x\in G$. In this section we study $\mathfrak{n}$-groups.
Following \cite{c} a finite group $G$ is called  an sn-group whenever $$S_G(x)=\{g\in G \;|\; \langle
x\rangle \;\;\textrm{is a subnormal subgroup of}\;\; \langle x, g\rangle \}$$ is a subgroup of $G$ for all $x\in G$ and also $G$ is called an sn($p$)-group (for some prime $p$) if
$S_G(x)\leq G$ for all $p$-elements $x\in G$. For example, the
dihedral group $D_{2n}$,  is not only an
$\mathfrak{n}$-group but also an sn-group.
 There is another related class of groups, called $E$-groups: A group $G$ is called an  $E$-group
 if $$E_G(x)=\{g\in G \;|\; [g, _nx]=1 \;\;\text{for some positive integer}\;\; n\}$$ is a subgroup of $G$ for every
$x\in G$;  and for a prime $p$,   $G$ is called an $E_p$-group if $E_G(x)\leq G$ for
all $p$-elements $x\in G$. These class of groups was defined and studied by Peng \cite{Pe}.
Finite sn-groups has been studied by  Casolo \cite{c} as well as Heineken \cite{He}. Note that $nil_G(x)\subseteq S_G(x) \subseteq E_G(x)$ for any group $G$ and $x\in G$. The relation between these two latter classes of groups is established in the following result.
\begin{thm} \rm(\cite{c}, Theorem 1.11\rm) \label{t2}
Let $G$ be a finite group. The following are equivalent:
\begin{enumerate} \item $G$
is an sn($p$)-group for every $p$ dividing
$|G|$.
\item   $G$ is an sn-group;
\item   G is an $E$-group;
\item  $G$ is an $E_p$-group for every $p$ dividing $|G|$.
\end{enumerate}
\end{thm}
The non-abelian finite simple sn-groups are classified.
\begin{thm}\rm(\cite{c}, Theorem 2.1\rm)\label{t3}
A non-abelian finite simple group is an sn-group if and only if it is
one of the following groups: $$\mathrm{PSL}(2, 2^n), \hspace{3mm}n\geq
2;\hspace{5mm} \mathrm{Sz}(2^{2m+1}),~~~~
  m\geq 1$$
\end{thm}
Besides $nil_G(x) \subseteq S_G(x)$ for any element $x$ of a group
$G$, there are sn-groups which are not $\mathfrak{n}$-groups.
\begin{lem}\label{s4}
The symmetric group $S_4$ of degree $4$ is an $\mathrm{sn}$-group which  is not
 an $\mathfrak{n}$-group.
 \end{lem}
  \begin{proof} We note that (e.g., by {\sf GAP} \cite{GAP}) that
  $$ nil_{S_4}\big((12)(34)\big)=\{(),(34),(13),(23),(24),(14),(12),(12)(34), (13)(24),$$
 $$(14)(23),(1324),(1432),(1234),(1243),(1342),(1423)\}.$$ Thus $S_4$ is not an $\mathfrak{n}$-group and  it follows from Lemma 1.15 of \cite{c} that $S_4$ is an $\mathrm{sn}$-group.
\end{proof}
For a prime number $p$, we call a group $G$ an $\mathfrak{n}_p$-group whenever $nil_G(x)$ is a subgroup of $G$ for all $p$-elements $x$ of $G$. We denote by  $\pi(x)$ the set of all primes  dividing of $|x|$ for any element $x$ in a finite group.
\begin{lem}\label{p-nil}
A finite group $G$ is an $\mathfrak{n}$-group if and only if $G$ is an $\mathfrak{n}_p$-group for all primes  $p$ dividing $|G|$.
\end{lem}
\begin{proof}
Suppose that $G$ is an $\mathfrak{n}_p$-group for all prime numbers $p$ dividing $|G|$. Let $x$ be a nontrivial element of $G$ and $k$ be
the number of prime divisors of $|x|$.   We argue by induction on $k$ to prove that $nil_G(x)$ is a subgroup of $G$. If $k=1$, then the proof
follows from the hypothesis. Thus we may assume by induction that $k>1$ and for every nontrivial element $a$  of $G$ with $|\pi(a)|<k$, $nil_G(a)$
 is a subgroup of $G$. Let $y,z\in nil_G(x)$. We have to prove that $K=\langle yz,x\rangle$ is nilpotent.
To prove that $K$ is nilpotent, it is enough to find a positive integer $m$ such that $[z_1,\dots,z_m]=1$ for all $z_i\in\{yz,x\}$.
 Since $|\pi(x)|\geq 2$, there are  nontrivial commuting elements $a$ and $b$ of $G$ such that $x=ab$, $a$ is an $p$-element and $b$ is an
 $p'$-element for some prime number $p$. Now by induction hypothesis $nil_G(b)\leq G$.
  Since both $a$ and $b$ are powers of $x$,  we have $y,z\in nil_G(a)\cap nil_G(b)$.  Hence, by hypothesis, $yz\in nil_G(a) \cap nil_G(b)$.
 Since $a$ is a $p$-element and $[yz,a]$ belongs to the nilpotent group $\langle yz,a\rangle$, it follows that $[yz,a]$ is also a $p$-element.
  Also $[yz,b]$ is a $p'$-element. Now since $\langle yz,a\rangle$ is nilpotent,  $$yz\in nil_G([yz,a]); \eqno{(1)}$$ and since $yz,a\in nil_G(b)$,
   we have $[yz,a]\in nil_G(b)$ and equivalently $$b\in nil_G([yz,a]). \eqno{(2)}$$ It follows from (1) and (2) that $[yz,b]\in nil_G([yz,a])$
   and so $\langle [yz,b],[yz,a]\rangle$ is nilpotent. Now since $\gcd(|[yz,a]|,|[yz,b]|)=1$, $[yz,a]$ commutes with $[yz,b]$. Also it follows
   from (2) that $b$ commutes with $[yz,a]$. Hence $$[yz,x]=[yz,ab]=[yz,b][yz,a]^b=[yz,a][yz,b].$$
   Now let $yz=cd$ where $c$ and $d$ are two  commuting elements of $G$ such that $c$ is an $p$-element and $d$ is an $p'$-element.
    Note that $c,d \in nil_G(a) \cap nil_G(b)$, since
   they are powers of $yz$. By a similar argument we may prove that $[yz,a]=[c,a]$, $[yz,b]=[d,b]$ and by continuing in this manner we have that
   $$[z_1,\dots,z_s]=[t_1,\dots,t_s][y_1,\dots,y_s]$$ where $z_i\in\{yz,x\}$, $t_i\in\{a,c\}$ and $y_i\in\{b,d\}$. Let $m$ be the maximum of
   the nilpotency classes of $\langle a,c\rangle$ and $\langle b,d\rangle$. Thus $$[t_1,\dots,t_s]=[y_1,\dots,y_s]=1 \;\;\text{for all}\; s> m.$$
  Therefore   $\langle yz,x\rangle$ is nilpotent of class at most $m$. This completes the proof of the ``if'' part.\\
   The converse is obvious.
\end{proof}
\begin{lem}\label{ab-nil}
Let $p$ be a prime number. If the Sylow $p$-subgroups of a finite group $G$ are abelian, then  $G$ is an $\mathfrak{n}_p$-group.
In particular, $nil_G(x)=C_G(x)$ for all $p$-elements $x\in G$.
\end{lem}
\begin{proof}
Let $y\in nil_G(x)$ and $y=y_1y_2$ where $y_1$ is an $p$-element, $y_2$ is an $p'$-element and $[y_1,y_2]=1$.
Then $\langle x,y\rangle=\langle x,y_1,y_2\rangle$ is nilpotent. Thus $[y_2,x]=1$ and as $\langle x,y_1\rangle$ is a
 $p$-group, it follows from the hypothesis that $[x,y_1]=1$. Therefore $[x,y]=1$ which means that $y\in C_G(x)$. This
  completes the proof, as $C_G(x)\subseteq nil_G(x)$.
\end{proof}
 A group $G$ is an AC-group if $C_G(g)$ is  abelian
  for  all $g\in G\setminus Z(G)$. We can now see that some classes of AC-groups are $\mathfrak{n}$-groups.
 \begin{lem}\label{lt1}
  If $G$ is a centerless (not necessarily finite) $\mathrm{AC}$-group,   then $G$
 is an $\mathfrak{n}$-group. In particular,  $nil_G(a)=C_G(a)$ for all nontrivial elements $a$ of $G$.
 \end{lem}
 \begin{proof}
 Suppose that $a\in G\setminus \{1\}$ and $b$ is an element of $G$ such that $\langle a, b\rangle$ is nilpotent. Then
$Z(\langle a, b\rangle)\neq 1$ and also $Z(\langle a,
b\rangle)\leq C_G(a)\cap C_G(b)$. Since $G$ is an $\mathrm{AC}$-group,
 $C_G(a)\cap C_G(b)=Z(G)$ or $ab=ba$. If $C_G(a)\cap
C_G(b)=Z(G)$,  we have $1\neq Z(\langle a, b\rangle)\leq
C_G(a)\cap C_G(b)=Z(G)=1$, a contradiction.  Thus  $ab=ba$, that is $nil_G(a)=C_G(a)$.
 \end{proof}
  Now we study some famous families of finite non-abelian simple groups for to be a $\mathfrak{n}$-group.  We
  show that every simple sn-group is an $\mathfrak{n}$-group.
  \begin{lem}\label{lr}
Let $G$ be a group and $H$  a nilpotent subgroup of $G$ in which
$C_G(x)\leq H$ for every $x\in H\backslash\{1\}$. Then $nil_G(x)=H$ for every
$x\in H\backslash\{1\}$.
\end{lem}
\begin{proof}
Suppose  $y\in nil_G(x)$. Then $\langle y, x \rangle$ is a
nilpotent subgroup. Thus
 there exists a nontrivial element $z\in Z(\langle y, x \rangle)$. Hence  $z\in
 C_G(x)\leq H$ and so $z\in H$. It follows
 that $y\in C_G(z)\leq H$. Hence  $nil_G(x)\leq H$. Now since
 $H$ is a nilpotent subgroup,   $nil_G(x)=H$. This completes the proof.
\end{proof}
\begin{thm}
\label{tt} Every simple $\mathrm{sn}$-group is an $\mathfrak{n}$-group.
\end{thm}
\begin{proof}
 By Theorem \ref{t3}, it is enough to show that  $G_1=\mathrm{PSL}(2,2^n)$ and
 $G_2=\mathrm{Sz}(2^{2m+1})$ are
 $\mathfrak{n}$-groups for $n\geq 2$ and $m\geq 1$.
The group $G_1$ is a centerless AC-group (see e.g.,  \cite[Proposition 3.21, case 7]{Ab2}). Thus by Lemma
\ref{lt1}, $G_1$ is an $\mathfrak{n}$-group.\\
 The Suzuki group $G_2$ contains subgroups $F, A, B$ and $C$
such that  the set
$$\mathcal{P}=\big\{A^x\backslash\{1\}, B^x\backslash\{1\}, C^x\backslash\{1\}, F^x\backslash\{1\} \;|\;
x\in G_2 \big\}$$ is a partition for $G_2\backslash\{1\}$   and $C_{G_2}(b)\leq M\cup\{1\}$ for all  $b\in M$
 and for every  $M\in \mathcal{P}$.\\
 Now let $a\in G_2\setminus \{1\}$. Since $\mathcal{P}$ is a partition of $G_2\backslash\{1\}$,  $a\in M$ for
  some $M\in \mathcal{P}$. Now Lemma \ref{lr}  implies that  $nil_{G_2}(a)=M\cup\{1\}$.
 Hence $G_2$ is an $\mathfrak{n}$-group. This completes the proof.
\end{proof}
\begin{lem}\label{q=5,3}
Let $q$ be a prime power number such that $q^2\not\equiv 1 \mod 16$. Then $G=\mathrm{PSL}(2,q)$ is an $\mathfrak{n}$-group. In particular, $nil_G(a)=C_G(a)$ for all nontrivial elements $a\in G$.
\end{lem}
\begin{proof}
If $q\in\{3,5\}$ or $q$ is a power of $2$, $G$ is an AC-group and so we are done by Lemma \ref{lt1}. Thus we may assume  $q>5$ and $\gcd(q,2)=1$.
Then it follows from \cite[Satz 8.10, Chapter II]{HuppertI} that all Sylow subgroups of $G$ are abelian
(note that the hypothesis $q^2\not\equiv 1 \mod 16$ is used to prove that the Sylow $2$-subgroups are elementary abelian).
Now the proof completes by Lemmas \ref{ab-nil} and  \ref{p-nil}.
\end{proof}
Note that by Lemma 2.3 of \cite{c}, groups stated in Lemma \ref{q=5,3} are sn($2$)-groups but by Theorem 2.1 of \cite{c}  they are not sn-groups.
\begin{lem}
Let $q$ be a prime power number such that $q^2\equiv 1 \mod 16$. Then $G=\mathrm{PSL}(2,q)$ is {\bf not} an $\mathfrak{n}$-group.
\end{lem}
\begin{proof}
By  \cite[8.18 Satz, chapter II]{HuppertI}, $G$ has a subgroup isomorphic to $S_4$. Now Lemma \ref{s4} completes the proof,
 since the class of $\mathfrak{n}$-groups is clearly closed under taking subgroups.
\end{proof}
\section{\textbf{Groups whose  non-nilpotent graphs has finite clique number}}
Before starting to show the results, we recall some concepts for a
simple graph $\Delta$. A path $P$ in $\Delta$ is a sequence
$v_{0}-v_{1}-\ldots -v_{k}$ whose terms are vertices of $\Delta$
such that for any $i\in \{1, \ldots, k\}$, $v_{i-1}$ and $v_{i}$
are adjacent. In this case $P$ is called a path between $v_{0}$
and $v_{k}$. The number $k$
 is called the length of $P$.
  If $v$  and $w$ are vertices in $\Delta$, then $d(v, w)$
denotes the length of the shortest path between $v$ and $w$ and we
call $d(v, w)$ the distance between $v$ and $w$ in $\Delta$. We
say that $\Delta$ is connected if there is a path between each
pair of the vertices of $\Delta$. If $\Delta$ is connected graph, then
 The largest distance between all pairs of the vertices of $\Delta$
is called the diameter of $\Delta$, and it is denoted by
$dim(\Delta)$ and also if $\Delta$ is disconnected graph, then we
 define the diameter of $\Delta$ as following:
$$dim(\Delta)=\max\{dim(\Delta_i)| \Delta_i ~~\text{is a  connected
 component of}~~ \Delta \}.$$
One of our motivations for associating with a group such kind of
graph is a problem posed by P. Erd\"{o}s: For a group $G$, consider a
graph $G$ whose vertex set is $G$ and join two distinct elements
if they do not commute. Then he asked: Is there a finite bound for
the cardinalities of cliques in $G$ , if $G$ has no infinite
clique?\\
B.H. Neumann \cite{ne}, answered positively Erd\"{o}s' problem by
proving that such groups are exactly the center-by-finite groups
and the index of the center can be considered as the requested
bound in the problem.  Wiegold and Lennox \cite{wie.l} proved
that if $G$ is a finitely generated solvable group in which every
infinite subset contains two distinct elements $x$ and $y$ such
that $\langle x, y\rangle$ is nilpotent, then $G$ is
finite-by-nilpotent. The following result is an easy consequence
of the latter which may be considered as an answer to a Erd\"{o}s
like question on non-nilpotent graphs.
\begin{thm}{\rm (\cite{wie.l}\rm)} If $G$ is a  finitely generated solvable
group, then $\mathcal{N}_G$ has no infinite clique if and only if
the clique number of $\mathcal{N}_G$ is finite.
\end{thm}
Let $n\in\mathbb{N}$. We say that a group $G$ satisfies the condition
$(\mathcal{N}, n)$ whenever in every subset with $n+1$ elements of
$G$ there exist distinct elements $x$, $y$ such that $\langle x,
y\rangle$ is a nilpotent group.  This simply means that a group $G$ satisfies the condition $(\mathcal{N},n)$ whenever  every clique of $\mathfrak{N}_G$ has
size at most $n$. Endimioni in \cite{en2}, proved that if $n\leq
20$, then every finite group satisfying the condition
$(\mathcal{N}, n)$ is solvable and moreover the
alternating group $A_5$ of degree $5$ satisfies the condition
$(\mathcal{N}, 21\rm)$. Also if $n \leq 3$,  every
finite group satisfying the condition $(\mathcal{N}, n)$ is
nilpotent. Now we can summarize the latter results in terms
of non-nilpotent graph as following.
\begin{thm} {\rm (\cite{en2})} A finite group $G$ is nilpotent if and only if  $\omega(\mathcal{N}_G)\leq 3$ and  if
$\omega(\mathcal{N}_G)\leq 20$, then $G$ is solvable.
\end{thm}
It is asked in \cite[Question 1]{AA} that:
\begin{que}
Is there a non weakly nilpotent group $G$ with
$\omega(\mathfrak{N}_G)\leq 3$?
\end{que}
Some partial answers to this question are given in \cite[Corollaire 2.2]{AA}.\\

Tomkinson in \cite{to}, proved that if $G$ is a finitely generated
solvable group satisfying the condition $(\mathcal{N}, n)$, then
$|\frac{G}{Z^*(G)}| < n^{n^4}$. This result gives a bound for the size of a finite
solvable centerless group satisfying the condition $(\mathcal{N},n)$.
\begin{thm}{\rm (\cite{to})} Let $G$ be a  finitely
generated solvable  group such that
$\omega(\mathcal{N}_G)$ is finite. Then the index of the
hypercenter of $G$ is finite and bounded above by
$(\omega(\mathcal{N}_G))^{\omega(\mathcal{N}_G)^4}$.
\end{thm}
In \cite{Ab3}  more properties of groups in
$(\mathcal{N}, n)$ have been studied. These results can be stated in terms of non-nilpotent graph
as follows.
\begin{thm} Let $G$ be a finite group without nontrivial normal abelian subgroups such that  $\omega(\mathcal{N}_G)$ is finite. Then $|G|$ is
bounded above by a function of $\omega(\mathcal{N}_G)$.
\end{thm}
\begin{thm} Let $G$ be a  finite non-solvable group. Then  $\omega(\mathcal{N}_G)= 21$ if and only if
$\frac{G}{Z^*(G)} \cong A_5$.
\end{thm}
\begin{thm} Let $G$ be a  finite  group. Then  $\omega(\mathcal{N}_G)= 4$ if and only if
$\frac{G}{Z^*(G)} \cong S_3$.
\end{thm}
Here we give some general properties of non-nilpotent graphs.
\begin{prop}\label{z3} Let $G$ be a  group which is not weakly nilpotent. Then
$\mathrm{diam}(\mathcal{N}_G)\geq 2$. Moreover,
$\mathrm{girth}(\mathcal{N}_G)=3$.
\end{prop}
\begin{proof} We show that there exists an element $x\in G\setminus nil(G)$ such that $x\not=x^{-1}$. Suppose, for a contradiction,  that  $a=a^{-1}$
for all $a\in G\setminus nil(G)$.  In this case, we prove $G=nil(G)$, which is a contradiction. Let $x,y\in G\setminus nil(G)$.
 If $xy\in nil(G)$, then $\langle x, y\rangle$  is nilpotent,
since $\langle xy, y\rangle=\langle x, y\rangle$. If $xy\not\in nil(G)$, then $xy=yx$, since $(xy)^2=x^2=y^2=1$.
It follows that  $\langle a, b\rangle$  is nilpotent for all $a,b\in G$.\\
Therefore there
exists an element $x\in G\setminus nil(G)$ such that $x\neq
x^{-1}$. Since $x\in G\setminus nil(G)$, there exists  an element $y\in G$ such that $\langle x,y\rangle=\langle x^{-1},y \rangle$ is not nilpotent.   Hence $x-y-x^{-1}$ is a path of length $2$
 and so $\mathrm{diam}(\mathcal{N}_G)\geq 2$.\\
Suppose that $x\in G\setminus nil(G)$.  Thus there exists $y\in
G\setminus nil(G)$ such that $x$ and $y$ are adjacent. It follows that
$x-xy-y$ is a cycle in $\mathcal{N}_{G}$. That is,
$\mathrm{girth}(\mathcal{N}_{G})=3$.
\end{proof}
\begin{prop}
There is no finite non-nilpotent group $G$ with a non-nilpotent
proper subgroup $H < G$ such that $\mathfrak{N}_G\cong
\mathfrak{N}_H$.
\end{prop}
\begin{proof}
Suppose, for contradiction, that there is a finite non-nilpotent
group $G$ with   a non-nilpotent proper subgroup $H$ satisfying
$\mathfrak{N}_{G}\cong \mathfrak{N}_H$. Then $|G|-|nil(G)| = |H| -
|nil(H)|$ and so by Proposition \ref{l0}, we have $|G|-|Z_k(G)| =
|H| - |Z_l(H)|$, where $k$ and $l$ are positive integers. Clearly,
$|H| \leq \frac{1}{2}|G|$ and $|Z_k(G)| < \frac{1}{2}|G|$. It
follows that $|G| = |H| - |Z_l(H)| + |Z_k(G)| < |G|$, which is a
contradiction.
\end{proof}
\begin{prop}
There is no finite non-nilpotent group $G$ with a
normal subgroup $N \leq G$ such that $G/N$ is non-nilpotent and  $\mathfrak{N}_G\cong
\mathfrak{N}_{\frac{G}{N}}$.
\end{prop}
\begin{proof}
Suppose, for contradiction, that there is a finite non-nilpotent
group $G$ with  a normal subgroup $N\leq G$ such that $G/N$ is
non-nilpotent and  $\mathfrak{N}_G\cong
\mathfrak{N}_{\frac{G}{N}}$. Then $|G|-|nil(G)| =
|\frac{G}{N}|-|nil(\frac{G}{N})|$ and so by Proposition \ref{l0},
we have $|G|-|Z_k(G)|=|\frac{G}{N}|- |Z_l(\frac{G}{N})|$, where
$k$ and $l$ are positive integers. Now by Case $1$ of Lemma
\ref{ll1} and Proposition \ref{l0},
$$Z_l(\frac{G}{N})\geq \frac{Z_l(G)N}{N}.$$ It follows that
$$|G|-|Z_k(G)|=|\frac{G}{N}|- |Z_l(\frac{G}{N})|\leq |\frac{G}{N}| -
|\frac{Z_l(G)N}{N}|=|\frac{G}{N}| - |\frac{Z_l(G)}{N\cap
Z_l(G)}|.$$ Let $|G|=m$, $|N|=n,$ $|Z_k(G)|=r_1$, $|N\cap
Z_l(G)|=t$, $|Z_l(G)|=r_2$ and $a=\frac{r_2}{r_1}$. It follows
that $m-r_1\leq \frac{m}{n}- \frac{r_2}{t}$ and so
$\frac{m}{r_1}\leq \frac{1-\frac{a}{t}}{1-\frac{1}{n}}\leq
\frac{1-\frac{a}{t}}{1-\frac{1}{5}}\leq 5$, since $n\geq 6$. Hence
$\frac{G}{Z_k(G)}$ is a nilpotent group and so $G$ is a nilpotent
group, a contradiction.
\end{proof}
 \begin{prop}
Let $G$ be a non-nilpotent finite solvable group. Then $$|E(\mathfrak{N}_G)|\geq \frac{p-1}{2p}|G|^2,$$ where $p$ is the
smallest prime number dividing the order of $G$.
\end{prop}
\begin{proof}
 By Theorem $5$ of  \cite{tt},  $v_0\geq
\frac{p-1}{p}$, where $v_0$ denotes the proportion of ordered
pairs of $G$ that generate a non-nilpotent subgroups.
 Therefore $v_0=\frac{2|E(\mathfrak{N}_G)|}{|G|^2}\geq \frac{p-1}{p}$ and so $|E(\mathfrak{N}_G)|\geq \frac{p-1}{2p}|G|^2$.
   \end{proof}
\section{\textbf{Connectedness of the non-nilpotent graphs}}
In this section study the connectedness of $\mathfrak{N}_G$ for a finite non-nilpotent group $G$. We prove that $\mathfrak{N}_G$ is connected and $\mathrm{dim}(\mathfrak{N}_G)\leq 6$. Moreover
 some results and examples suggest that $\mathrm{dim}(\mathfrak{N}_G)=2$ for any finite non-nilpotent group $G$.
\begin{thm}
Let $G$ be a finite non-nilpotent group. Then $\mathfrak{N}_G$ is
connected and its diameter is at most $6$. In particular, every
two vertices $x$ and $y$ with $\pi(x)\not=\pi(y)$ are connected by
a path of length at most $4$.
\end{thm}
\begin{proof} We first prove that if $x$ and $y$ are two distinct vertices of $\mathfrak{N}_G$ which are of prime power orders, then there is a path
connecting  them of length  at most $6$.\\
Suppose $x$ and $y$ are two $p-$ and $q$-elements of $G\backslash nil(G)$, where  $p$ and $q$ are two distinct prime numbers.
 We prove that there is a path connecting $x$ and $y$ of length at most $4$.  We may assume that $x$ and $y$ are not adjacent.
  This means that $\langle x,y\rangle$ is nilpotent and since $\gcd(|x|,|y|)=1$, we have  $\langle x,y\rangle=\langle xy\rangle$.
   Now since $x,y\not\in nil(G)$, there are elements $x'$ and $y'$ which are adjacent to $x$ and $y$, respectively.
   Since $\langle xy,x'\rangle=\langle x,y,x'\rangle$ and $\langle x,x'\rangle$ is not nilpotent, it follows that $xy$ and $x'$
   are adjacent. By a similar argument, $xy$ is also adjacent to $y'$ and so $x-x'-xy-y'-y$ is a path of length $4$.\\
Now suppose that $x$ and $y$ are two $p$-elements of $G\backslash
nil(G)$ for some prime $p$. Since $G$ is not nilpotent, there
exists an $q$-element $z\in G\backslash nil(G)$ such that
$p\not=q$. Since $x,y,z$ are not in $nil(G)$, there are elements
$x',y',z'$
 which are adjacent to $x,y,z$, respectively. By a similar  proof $$x-x'-xz-z'-yz-y'-y$$ is a path of length  $6$ connecting $x$ and $y$.\\
Now let $x$ and $y$ be two arbitrary elements in $G\backslash
nil(G)$. Then there are integers $n$ and $m$ such that
$x^n,y^m\not\in nil(G)$ and they are of prime power orders.  Thus
by the above, there is a path $x^n-x_1-\cdots-x_k-y^m$ such that
$k\leq 3$ if $\gcd(|x^n|,|y^m|)=1$ and $k\leq 5$ if
$\gcd(|x^n|,|y^m|)\not=1$. Since $\langle x^n,x_1\rangle \leq
\langle x,x_1\rangle$ and $\langle y^m,x_k\rangle \leq \langle
y,x_k\rangle$, $x$ is adjacent to $x_1$ and $y$ is adjacent to
$x_k$. Hence $x-x_1-\cdots-x_k-y$ is a path
 of length $k+1$ connecting $x$ to $y$. Note that if $\pi(x)\not=\pi(y)$, then we may choose $n$ and $m$ such that $\gcd(|x^n|,|y^m|)=1$ and so
  in this case $x$ and $y$ are connected with a path of length at most $4$. \\
\end{proof}
\begin{que}
What is
$\mathfrak{d}:=\max\big\{\mathrm{diam}\big(\mathfrak{N}_G\big)
\;|\; G \;\text{is a finite non-nilpotent group}\big\}$?
\end{que}
 By the  theorem above $\mathfrak{d}\leq 6$.
\begin{thm}Let $G$ be a finite group having no nontrivial normal abelian subgroup. Then
$\mathrm{diam}(\mathfrak{N}_G)=2$.
\begin{proof}
  Since $G$ has no nontrivial normal abelian subgroup, the largest normal solvable subgroup $S(G)$ of $G$ is trivial. Since a nontrivial nilpotent group has nontrivial center, it follows from  Lemma
  \ref{l0} that $nil(G)=1$ . Now also by Theorem 6.4 of \cite{G.B.S}, for every two  nontrivial elements  $x$ and $y$ in $G$ there
  exists   an element $z$ in $G$ such that $\langle x, z\rangle$ and $\langle y, z\rangle$ are not
 solvable subgroups and so are not nilpotent subgroups. This implies that the graph  $\mathfrak{N}_G$  is
   connected and  $diam(\mathfrak{N}_G)=2$.
\end{proof}
\end{thm}
\begin{prop}
Let $G$ be an $\mathfrak{n}$-group which is not weakly nilpotent.
Then $\mathfrak{N}_{G}$ is a connected graph and
$\mathrm{diam}(\mathfrak{N}_G)=2$.
\end{prop}
\begin{proof}Suppose, for a contradiction, that $diam(G)>2$ (note that by Proposition
 \ref{z3}, $diam(\mathfrak{N}_G)\neq
1$). Thus there exist two vertices $x$ and $y$ of
$\mathfrak{N}_G$ such that $d(x, y)\neq 2$, and so $G=nil_G(x)\cup
nil_G(y)$. Now since $G$ is an $\mathfrak{n}-$group, it follows
that either $G=nil_G(x)$ or
 $nil_G(y)=G$. This gives a contradiction, as $x$ and $y$ are not in $nil(G)$.
\end{proof}
\section{\textbf{Groups whose non-nilpotent graphs are planar}}
A planar graph is a graph which can be drawn in the plane so that
whose edges intersect only at end vertices. Note that  every subgraph of a planar graph is also planar.
\begin{thm}\label{z2}Let $G$ be a finite non-nilpotent group. Then
 $\mathfrak{N}_G$
is planar if and only if $G\cong S_{3}$.
\end{thm}
\begin{proof}
Since the Engel graph $\mathcal{E}_{G}$ is a subgraph of
$\mathfrak{N}_G$ (see \cite{Ab1} for the definition),  the
 graph $\mathcal{E}_{G}$ is a planar graph. It follows from Theorem
 3.1 of  \cite{Ab1} that $G$ is isomorphic to one of the following groups:
$$S_{3},~~~ D_{12} \; \; \text{or} \;\;\;\;\; T=\langle x, y \mid
x^{6}=x^{3}y^{-2}=x^{y}x=1 \rangle.$$ On the other hand by Lemmas
\ref{ab-nil} and  \ref{p-nil}, the groups $T$ and $D_{12}$ are $\mathfrak{n}$-groups
and $nil_G(a)=C_G(a)$, for all $a\in T$ or $D_{12}$. Therefore for
these cases we have $\Gamma_{G}\cong \mathfrak{N}_G$, where $\Gamma_G$ is the commuting graph of $G$ (see \cite{Ab2} for the definition). Now
Proposition 2.3 of \cite{Ab2}  implies that $\Gamma_{G}$ is
planer if and only if $G\simeq S_3$, $D_8$ or $Q_8$. It follows
that $G\cong S_3$. The converse is clear.
\end{proof}
\begin{cor}
Let $G$ be a non-nilpotent group such that $\mathfrak{N}_G\cong
\mathfrak{N}_{S_{3}}$, then $G\cong S_{3}$.
\end{cor}
\begin{proof}
It follows from Theorem \ref{z2}, that the only non-nilpotent
group whose non-nilpotent graph is planer, $S_{3}$. Hence $G\cong
S_{3}$.
\end{proof}

\section{\textbf{Groups whose non-nilpotent graphs are regular}}
 In this section we give a characterization of finite nilpotent groups in terms of non-nilpotent graphs.  In this section we prove that
 \begin{thm}\label{ttt}
A finite group $G$  is  nilpotent  if and only if the set of
vertex degrees of $\mathcal{N}_G$ has at most two elements.
\end{thm}
  Note that the elements in $nil(G)$ has degree $0$ (isolated vertices) in $\mathcal{N}_G$; so what we want to prove in this section is equivalent to this statement: There is no finite non-nilpotent group $G$ such that $\mathfrak{N}_G$ is a regular graph.\\

 Throughout  this section, suppose,  for a contradiction, that  $G$ is  a finite non-nilpotent group such that $\mathfrak{N}_G$ is
  regular.  Thus  $|nil_G(x)|=|nil_G(y)|$
 for all $x, y \in G \setminus nil(G)$. Note that, by Lemma \ref{ll1},
 the non-nilpotent
 graph of $H=\frac{G}{Z^{*}(G)}$ is also regular and
 $Z^{*}(H)=nil(H)=1$. Therefore we may assume that $nil(G)=1$ and $|nil_G(x)|=|nil_G(y)|$
  for any two nontrivial elements $x, y$ of $G$.
   \begin{lem}\label{l8}
 For all nontrivial elements $a \in G$, we have
  $nil_G(a)=nil_G(a^{i})$ for all $1\leq i\leq |a|-1$
 In particular, $\langle a, x \rangle$ is not  nilpotent if and only if
 $\langle a^{i}, x \rangle$ is   not nilpotent for some $i\in\{1,\dots,|a|-1\}$.
  \begin{proof}
It is clear that $nil_G(a)\subseteq nil_G(a^{i})$. Thus $nil_G(a)= nil_G(a^{i})$ since $|nil_G(a)|=|nil_G(a^{i})|$ for all  $i\in\{1,\dots,|a|-1\}$.
 \end{proof}
\end{lem}
 \begin{lem}\label{l5}
 For all nontrivial elements $a, b \in G$ such that
 $ab=ba$ we have the following:
 \begin{enumerate}
   \item If $|a|<|b|$ then  $nil_G(ab)= nil_G(b)$.
  \item If $\gcd( |a|,|b|)=1$, then
  $nil_G(a)=nil_G(b)=nil_G(ab)$.
  \end{enumerate}
   \begin{proof}
 (1) \;  Suppose that $c\in nil_G(ab)$, it follows that $\langle c,
ab\rangle$ is a nilpotent subgroup, so  $\langle c, (ab)^{\mid a
\mid}\rangle=\langle c, b^{\mid a \mid}\rangle$ is a nilpotent
subgroup and so by Lemma \ref{l8}, $\langle c, b\rangle$ is a
nilpotent subgroup, that is  $c\in nil_G(b)$.
Thus $nil_G(ab)\subseteq nil_G(b)$ and so $nil_G(ab)= nil_G(b)$, since  $|nil_G(ab)|=|nil_G(b)|$. \\
(2)\; It is enough to $nil_G(ab)\subseteq nil(a)\cap nil_G(b)$.
Let $c\in nil_G(ab)$, so $\langle c, ab\rangle$ is a nilpotent
subgroup. It follows that $\langle c, (ab)^{\mid a\mid}\rangle
=\langle c, b^{\mid a\mid}\rangle$ and $\langle c,(ab)^{\mid
b\mid}\rangle = \langle c, a^{\mid b\mid}\rangle$ are nilpotent
subgroups and so two subgroups $\langle c, b\rangle$ and $\langle
c, a\rangle$ are nilpotent, since $nil_G(b)=nil_G(b^{\mid a
\mid})$ and $nil_G(a)=nil_G(a^{\mid b \mid})$. Thus $c\in
nil_G(a)\cap nil_G(b)$, namely $nil_G(ab)\subseteq nil_G(a)\cap
nil_G(b)$. Therefore by Case $(1)$ and regularity,
 $nil_G(ab)= nil_G(a)= nil_G(b)$.
  \end{proof}
\end{lem}
\begin{lem}\label{hall}
If $a$ is an element of $G$ which is not of prime power order, then $nil_G(a)$ is a nilpotent Hall subgroup of $G$.
\end{lem}
\begin{proof}
By hypothesis, there are integers $i$ and $j$ such that $|a^i|=p$ and $|a^j|=q$ are distinct prime numbers.
 By Lemma \ref{l8}, $nil_G(a^i)=nil_G(a^j)=nil_G(a)$. Let $y_1$ and $y_2$ be two nontrivial $p'$-elements of $nil_G(a^i)$.
  Thus $[y_1,a^i]=[y_2,a^i]=1$ and so $nil_G(a^i)=nil_G(y_1)=nilG(y_2)$. Therefore $\langle y_1y_2,y_2\rangle=\langle y_1,y_2\rangle$ is
  nilpotent. It follows that $y_1y_2\in nil_G(a)$ and so $$A=\{~y\in nil_G(a)~|~ y ~ \text{is an}~
p'\text{-element} \},$$ is a subgroup and as every pair of elements of $A$  generates a nilpotent subgroup, $A$ is a nilpotent $p'$-group.
Now let $z_1,z_2$ be two nontrivial $p$-elements of $nil_G(a^j)$. Then by a similar argument $nil_G(a)=nil_G(z_1)=nil_G(z_2)$ and so
$$B=\{~y\in nil_G(a)|~~\text{y is an}~ p-\text{element}\}$$ is $p$-subgroup. Also we have that $nil_G(a)=nil_G(x)=nil_G(y)$ for any two nontrivial elements $x\in A$ and $y\in B$. Hence $\langle x,y\rangle$ is nilpotent and so $[A,B]=1$. It is clear that $nil_G(a)\subseteq AB$ and so it follows that $nil_G(a)=AB$ is a nilpotent group. \\ It remains to prove $nil_G(a)$ is a Hall subgroup of $G$. Let $r$ be a prime number dividing $|nil_G(a)|$.
Then either $r\not=p$ or $r\not=q$. Assume without loss of generality that $r\not=p$ and let $x$ be a nontrivial $r$-element of $nil_G(a)$. Note that
$nil_G(x)=nil_G(a)$, by the first part of the proof. Now if $Q$ is any Sylow $r$-subgroup of $G$ containing $x$, then $Q\leq nil_G(x)$. It follows that $nil_G(a)$ is a Hall subgroup of $G$.
\end{proof}
\begin{lem}\label{CP}
Every element of $G$ has prime power order.
\end{lem}
\begin{proof}
Suppose, for a contradiction, that there is an element of $G$ which is not of prime power order. Then by Lemma \ref{hall}, $nil_G(a)$ is a nilpotent Hall subgroup of $G$. Since $G$ is assumed to be non-nilpotent, there is a prime number $p$  dividing $|G|$ and $p \nmid |nil_G(a)|$. Let $x$ be nontrivial $p$-element of $G$. Then by Lemma \ref{l9}, $p$ divides $|nil_G(x)|=|nil_G(a)|$, a contradiction. This completes the proof.
\end{proof}
 \begin{lem}\label{Cpp}
 Let $x$ be a nontrivial $p$-element of
$G$ for some prime $p$. Then $$nil_G(x)=\bigcup\{P \;|\; P \;\textrm{is a Sylow $p$-subgroup containing}\;\; x\}.$$ In particular, if a Sylow $p$-subgroup of $G$ is abelian, then $nil_G(x)=C_G(x)$ and so $nil_G(x)$ is a $p$-subgroup.
\end{lem}
\begin{proof}
 Suppose that  $y \in
nil_G(x)$. Thus $\langle x, y\rangle$ is nilpotent and so it follows from Lemma \ref{CP} that
$\langle x, y\rangle$ is a $p$-subgroup. Hence $\langle x, y\rangle$ is contained in a Sylow $p$-subgroup of $G$.
 This completes the proof.
\end{proof}
\begin{lem}\label{sol-nor}
$G$ is neither  solvable nor having a normal Sylow $p$-subgroup for some prime $p$ dividing $|G|$.
\end{lem}
\begin{proof}
Suppose, for a contradiction, that $G$ is either solvable or having a normal Sylow $p$-subgroup for some prime $p$ dividing $|G|$. It follows from Lemma \ref{CP} and \cite[Theorem 1]{Hig} that $G$  contains either a cyclic Sylow $q$-subgroup  for some prime  $q$ dividing $|G|$ or a normal Sylow $p$-subgroup  for some prime $p$ dividing $|G|$. Thus, by Lemma \ref{Cpp}, in any case there exists an element $x\in G$ such that $nil_G(x)$ is (a subgroup) of prime power order. Since $G$ is not nilpotent, $|G|$ is dividing by at least two primes. Now Lemma \ref{l9} gives a contradiction.
\end{proof}
\noindent{\bf Completion of the Proof of Theorem \ref{ttt}.}
By Lemma \ref{sol-nor}, $G$ is neither solvable nor having a normal Sylow $p$-subgroup for some prime $p$ dividing $|G|$. Now using the classification of finite groups in which every element has prime power order (see \cite{wbg} or \cite[{\sc Main Theorem}]{del}), it is easy to see that $G$ has a cyclic Sylow $q$-subgroup for some prime $q$ dividing $|G|$ (one must only check the cases (4) and (5) of \cite[{\sc Main Theorem}]{del} for this; the other cases are ruled out by Lemma \ref{sol-nor}). Now Lemmas \ref{Cpp} and \ref{l9} complete the proof.\\
The converse is clear. $\hfill \Box$

\noindent{\bf Acknowledgements.}   The research of the first author was partially supported by  the Center of Excellence for Mathematics, University of Isfahan and he gratefully acknowledges the financial support of University of Isfahan for the sabbatical leave studies in University of Bath, UK and ICTP, Trieste, Italy.

\end{document}